\documentclass[reqno]{amsart}
\usepackage{hyperref}
\usepackage{amssymb,amsthm,amsmath,amsfonts,,,enumerate,float,}
\usepackage[english]{babel}
\usepackage{blindtext}
\usepackage{amsthm,amssymb,amsmath}
\usepackage{subcaption,graphicx,float,caption,tcolorbox}
\usepackage{arydshln,tcolorbox}
\usepackage{tikz}
\usepackage{tikz-cd}
\usepackage{csquotes}
\date{}

\begin{document}
	\title[Ascent and Descent of Composition operators on Orlicz-Lorentz Spaces]{Ascent and Descent of Composition operators on Orlicz-Lorentz Spaces }
	
		\author{Neha Bhatia}
	\address{Neha Bhatia\newline Department of Mathematics, University of Delhi, Delhi 110007, India}
	\email{nehaphd@yahoo.com}

\author{ Anuradha Gupta}
\address{Anuradha Gupta\newline Department of Mathematics, Delhi College of Arts and Commerce, University of Delhi, Delhi 110023, India}
 \email{dishna2@yahoo.in}
 
 	\thanks{Submitted Month Day, 20xx. Published Month Day, 20xx.}


\begin{abstract} The aim of this paper is to discuss the  characterizations for the composition operator on Orlicz-Lorentz spaces $L_{\phi,w}$ to have finite ascent (descent).
	\end{abstract}
	\maketitle \numberwithin{equation}{section}

\newtheorem*{remark}{Remark}

\newtheorem{definition}{Definition}
\newtheorem{corollary}{Corollary}[section]
\newtheorem{theorem}[corollary]{Theorem}
\newtheorem{lemma}[corollary]{Lemma}
\newtheorem{example}{Example}

	\section{Introduction}
	Let $V$ be a vector space and $T$ be a linear operator on $V$. Then $N(T)$ and $R(T)$ denotes the null space and the range space of $T$ given by
	
	$$ N(T)=\{v\in V: T(v)=0\}$$ and
	$$R(T)=\{T(v): v\in V\}.$$
	
	For $m\geq 0$, the null space and range space of $T^m$ satisfies
	\begin{align*}
		N(T)\subseteq N(T^2)\subseteq N(T^3)\subseteq \cdots N(T^m)\subseteq N(T^{m+1})\cdots,\\
		R(T)\subseteq R(T^2)\subseteq R(T^3)\cdots \subseteq R(T^m)\subseteq R(T^{m+1}) \cdots
	\end{align*}
	
	\begin{definition}
		  If  there  is  an integer  $m>0$  such  that   $N(T^m) = N(T^{m+1})$, the  smallest  such  integer  is  called  the  ascent  of  $T$  and we  denote it   by  $\mathcal{A}(T)$.    If  no  such  integer $m$ exists such that $N(T^m) = N(T^{m+1})$,  we say  that ascent of $T$ is infinite i.e., $\mathcal{A}(T)=\infty$.
	\end{definition}

	\begin{definition}
		 If  there  is  an  integer  $m>0$  such  that  $R(T^{m+1})= R(T^m)$,    the  smallest  such  integer  is  called  the  descent  of  T  and  is  denoted  by  $\mathcal{D}(T)$.   If  there is no  such  integer $m$ such that $R(T^{m+1})= R(T^m)$, we say  that  descent of $T$ is infinite i.e., $\mathcal{D}(T)=\infty$.
	\end{definition}
	
	Many authors have studied the characterizations of composition and weighted
	composition operators with ascent and descent on different function spaces such as
	$L^p$ space, $l^p$ space, Orlicz space, Lorentz space [\cite{2}, \cite{4}, \cite{5}, \cite{6}]. Motivated by their work,
	we have discussed the characterization of composition operators on Orlicz-Lorentz
	spaces $L_{(\phi,\omega)}$ to have finite ascent and descent.

	\section{Composition Operators on Orlicz-Lorentz Spaces}
	Let $(E,\mathcal{E},\nu)$ be a $\sigma$-finite measure space and
	$f$ be any complex-valued measurable function on the space.
	
	For $s \geq 0$, define $\nu_{f}$ the \textit{distribution function} as
	\begin{align*}\nu_{f}(s) = \nu\{ x \in E :|f(x)| > s \}\end{align*}
	Clearly, $\nu_f$ is decreasing.	By $f^*$ we mean the \textit{non-increasing rearrangement} of $f$ given as
	\begin{align*}f^*(t) =\inf \{ s > 0 :\mu_{f}(s) \leq t\},\quad   t\geq 0\end{align*}
	$f^*$ is a non-negative and decreasing function.
	
	By a weight function $\omega$, we mean
	$\omega :J \rightarrow  J$ where $J = (0 \ \infty)$ a non-increasing locally integrable function
	such that $\int_{0}^{\infty}\omega(t) dt =\infty$.
	
	Let $\phi : [0 \ \infty) \rightarrow [0 \ \infty)$ be a convex function such that
	\begin{align*}&\phi(x) = 0 \Leftrightarrow x=0, \\ &\lim_{x\rightarrow \infty}\phi(x) = +\infty.\end{align*}
	
	Such a function is called a \textit{young function}.\ The young function is
	said to satisfy the $\Delta_2$-condition if for some $k >  0$,
	\begin{align*} \phi(2x) \leq k \phi(x),\quad\text{for all } \ x > 0\,. \end{align*}
	If $\phi$ satisfies $\Delta_2$-condition, then we define the space $L_{\phi,\omega}$  as
	\begin{align*} \bigg\{f : E \rightarrow \mathbb{C} \ \text{measurable} :\int_{0}^{\infty} \phi(\alpha f^*(t))\omega(t) dt <\infty \ \text{for some } \alpha > 0 \bigg\} \end{align*}
	The space $L_{(\phi,\omega)}$ is called an Orlicz-Lorentz space and is a Banach space with respect to
	the Luxemburg norm given by
	\begin{align*} \| f\|_{\phi,\omega} =\inf \bigg\{\epsilon > 0: \int_{0}^{\infty} \phi(|f^*(t)|/\epsilon)\omega(t)dt \leq 1\bigg\} \end{align*}
	
	Orlicz-Lorentz space is common generalization of Orlicz space and Lorentz space.
	If $w(t) = 1$ then it becomes an Orlicz space. If $\phi(x) = x^p$, $1 \leq  p < \infty$, then it
	becomes a Lorentz space $L_{(p,q)}$. Let $\mathcal{B}(L(\phi,\omega))$ denote the Banach algebra of all bounded
	linear operators on $L(\phi,\omega)$.
	For more details on Orlicz-Lorentz spaces, we refer to \cite{1}, \cite{7}.
	
	Let $\Psi:E\rightarrow E$ be a measurable non-singular transformation (i.e., $\nu(\Psi^{-1}(S))=0$ whenever
	$\nu(S)=0$ for $S\in\mathcal{E})$.
	
	If $\Psi$ is non-singular, then we say that $\nu \Psi^{-1}$ is absolutely continuous with respect to
	$\nu$.\ Hence, by Radon-Nikodym
	theorem, there exists a unique non-negative measurable function $g$ such that
	\begin{align*}
		(\nu \Psi^{-1})(S)=\int_{S} g d\nu,\quad \text{for} \ S \in \mathcal{E}.
	\end{align*}
	$g$ is called the Radon-Nikodym derivative and is denoted by $\dfrac{d\nu \Psi^{-1}}{d\nu}$. For $m\geq 2$, we observe that $\nu_{\Psi}^m \ll \nu_{\Psi}^{m-1} \ll \cdots \nu_\Psi$. Hence
	\begin{align*}
		\nu_\Psi^m=\int_S f_\Psi d \nu \ for \ S \in \mathcal{E}.
	\end{align*}

\begin{definition}
	Measures $\nu_1$ and $\nu_2$ are said to be equivalent if $\nu_2 \ll \nu_1 \ll \nu_2$.
\end{definition}
\begin{definition}
	A measurable transformation $\Psi$ is said to be measure preserving if it preserves the measure i.e., $\nu(\Psi^{-1}(A))= \nu(A)$ for all $A \in \mathcal{E}$.
\end{definition}

\begin{definition} Let $({E}, \mathcal{E}, \nu)$ be a measure space. A measurable transformation $\Psi:E \to E$ is said to be \textit{pre-positive} if it satisfies the condition $\nu(\Psi^{-1}(A)) > 0$ whenever $\nu(A) > 0$.\end{definition}	

	Let $\Psi$ be a measurable transformation then the  composition
	operator $C_{\Psi}$ \cite{8} induced by 
	$\Psi$ is given by
	\begin{align*}
		C_{\Psi}= f\circ \Psi,\quad\text{for every} \ f\in L_{(\phi,w)}.
	\end{align*}

	Let $\Psi:E \rightarrow E$ is a non- singular measurable transformation, then $\Psi^m$ is also a non-singular measurable transformation for every non-negative integer $m$ with respect to the measure $\nu$. Thus, we can define the composition operator $C_{\Psi^m}$ on Orlicz-Lorentz space $L_{(\phi,w)}$, such that
	$ C_\Psi^m(f_\Psi) = f_\Psi \ o \ \Psi^m = C_{\Psi^m}(f_\Psi)$ for every measurable function $f_\Psi$ of the Orlicz-Lorentz space. Also,  the measure $\nu o \Psi^{-m}$ is defined as
	$$\nu o \Psi^{-m}(S)=\nu o \Psi^{-(m-1)}(\Psi^{-1}(S))=\nu o \Psi^{-1}(\Psi^{-(m-1)}(S)) \ for \ S \in \mathcal{E}. $$
	Then
	\begin{align}
	\ldots \ll \nu \ o \ \Psi^{-(m+1)} \ll \nu \ o \ \Psi^{-m} \ll \nu \Psi^{-(m-1)} \ll \ldots \ll \nu \ o \ \Psi^{-1} \ll \nu.\label{eq:2.1}
	\end{align}
	
	If we put $\nu_m= \nu \ o \ \Psi^{-m}$, then by Radon - Nikodyn theorem, there exists a non-negative locally integrable function $f_{\Psi^{m}}$ on $E$ satisfying
	\begin{align}
		\nu_m(S)=\int_S f_{\Psi^m}(x) d\nu (x) \ for \; \ S \in \mathcal{E}.\label{eq: 2.2}
	\end{align}
	
	Here, $f_{\Psi^m} \left(=\frac{d\nu_m}{d\nu}\right)$
	is called the Radon -Nikodym derivative of $\nu_m$ with respect to $\nu$. 
	
	In \cite{1}, the necessary and sufficient condition for the boundedness of the composition operators on Orlicz-Lorentz spaces $L_{\phi,w}$ are discussed and given by the following theorem:
	\begin{theorem}
		Let $\Psi : E \rightarrow E$ be a non-singular measurable transformation. Then the composition operator $C_\Psi$  is bounded on $L_{\phi,w}$ if and only if there
		exists a constant $K>0$ such that 
		$$ \nu \Psi^{-1}(A) \leq K \nu(A)\  for\ A \in \ \mathcal{E}.$$
	\end{theorem}
	Also, the measurable transformation $\Psi$ is said to be bounded away from zero if
	there exists a positive real number $\epsilon$, such that $f(x) \geq  \epsilon$  for almost all $x \in S$ where
	$S = \{ x \in  E:  f(x) \neq 0 \}.$
	
		\section{Ascent of the Composition Operators}
	\begin{theorem}
		Let $C_\Psi$ be a composition operator on $L_{(\phi,w)}$. Then $\mathcal{N}(C_\Psi^m)= L_{(\phi,w)}(E_m)$ for $m \geq 0$ i.e.,  the kernel is  the collection of all the measurable functions $f_\Psi$ in $L_{(\phi,w)}$ satisfying $f_\Psi(x)=0$ for $x \in E / E_m$, where  $E_m= \{ x \in E:f_{\Psi^m}(x)=0 \}.$
	\end{theorem}
\begin{proof}
	Let $f$ be an element of $L_{(\phi,w)}(E_m)$. Then
	$$\nu \{ x \in E : f\ o \ \Psi^ m(x) \neq 0 \} \leq \nu \ o \ \Psi^{-m}(E_m) = \int_{E^{m}} f_ { \Psi^{m}}(x) d\nu (x) = 0.$$
	This implies that $f\circ \Psi^m=0$ a.e. This gives that 	$f_\Psi \in \mathcal{N}(C_\Psi^m)$. This implies that
	\begin{align}
	L_{(\phi,w)}(E_m) \subseteq \mathcal{N}(C_\Psi^m). \label{eq:3.1}
	\end{align}
Also, let  $f_\Psi \in \mathcal{N}(C_\Psi^m)$, then $\nu \ o \ \Psi^{-m}\{x \in E: f_\Psi(x) \neq 0 \} = 0.$ Take $F=\{ x \in E/E_m : f_\Psi(x) \neq 0 \}$ and $G=\{x \in E_m : f_\Psi(x) \neq 0\}$.
	From (\ref{eq: 2.2}), we have
\begin{align*}
	0=\int_F f_{\Psi^{m}}(x) d\nu (x) + \int_G {f_{\Psi^{m}}}(x)d\nu(x)\\
	=\int_F h_{\Psi^{m}}(x) d\nu (x)\\
	 \geq \frac{1}{n} \int_{F_n \cap F} d\nu\\
	  = \frac{1}{n} \nu ( F_n \cap F)
\end{align*}
 for each $n$ where $F_n=\{x \in E /E_m: f_\Psi>\frac{1}{n}\} \subseteq F$. Since $F_n \cap F=F$, this gives that $\nu(F_n \cap F)=\nu(F)$ for each $n$. Therefore $\nu(F)=0$. This implies that $f=0$ a.e. on $E /E_m$. Hence $ f\in L_({\phi,w})(E_m)$. This gives that
 \begin{align}
 	\mathcal{N}(C_\Psi^m)\subseteq	L_{(\phi,w)}(E_m). \label{eq:3.2}
 \end{align}
From (\ref{eq:3.1}) and (\ref{eq:3.2}), we obtain
\begin{align*}
		\mathcal{N}(C_\Psi^m)=	L_{(\phi,w)}(E_m).
\end{align*}
\end{proof}


\begin{lemma}
	Let $\Psi$ be a non-singular measurable transformation on the measure space $(E,\mathcal{E},\nu)$ that induces the composition operator $C_\Psi$ on the Orlicz-Lorentz space $L_{(\phi,w)}$. Then $N(C^m_\Psi)=N(C^{m+1}_\Psi)$ if and only if $\nu_m$ and $\nu_{m+1}$ are equivalent.
\end{lemma}

\begin{proof}
 Let $C_\Psi \in \mathcal{B}(L_{(\phi,w)})$. Suppose that $\nu_m$ and $\nu_{m+1}$ are equivalent.
Then $\nu_{m+1} \ll \nu_m \ll \nu_{m+1}$. So from $(\ref{eq:2.1})$, we have $\nu_m \ll \nu_{m+1} \ll \nu$ and $\nu_{m+1} \ll \nu_m \ll \nu$, and hence, by the chain rule,
\begin{eqnarray*}
	f_{\Psi^m}(x)&=&\frac{d\nu_m}{d\nu_{m+1}}(x).f_{\Psi^{m+1}}(x)\\
	f_{\Psi^{m+1}}(x)&=&\frac{d\nu_{m+1}}{d\nu_m}(x).f_{\Psi^m}(x)
\end{eqnarray*}
This gives that $E_m=E_{m+1}$. Since $\mathcal{N}(C_\Psi^m)=	L_{(\phi,w)}(E_m)$ for $m \geq 0$, therefore $$\mathcal{N}(C_\Psi^m)=	L_{(\phi,w)}(E_m)=L_{(\phi,w)}(E_{m+1})=\mathcal{N}(C_\Psi^{m+1}).$$

Conversely, suppose that $N(C_\Psi^m)=N(C_\Psi^{m+1})$. This gives $L_{\phi,w}(E_m)=L_{\phi,w}(E^{m+1})$. We first claim that $\nu(E_m/E_{m+1})=0$. The assumption of $\nu(E_m/ E_{m+1})>0$ provides a set $F_n=\{x\in E_m: f_{\Psi^{m+1}}(x)>\frac{1}{n}, n \in \mathbb{N}\}$ of non-zero finite measure. Characteristic functions $\chi_S$ are dense in Orlicz-Lorentz space for
$S \in \mathcal{E}$.  Now,
\begin{align*}
	\| \chi_{F_{n}}\|=\frac{1}{\phi^{-1}(\frac{1}{\nu(F_n)})}.
\end{align*}
Thus, $\chi_{F_n} \in L_{(\phi,w)}(E_m)= L_{(\phi,w)}(E_{m+1})$
i.e., $\chi_{F_n}$ vanishes outside $E_{m+1}$ which gives that $F_n \subseteq E_{m+1}$. Therefore,
\begin{align*}
	0 \leq \frac{1}{n}\int_{F_n} \chi_{F_n} d \nu \leq \int_{F_n} f_{\Psi^{m+1}} d \nu =0.
\end{align*}
This implies that $\nu(F_n)=0$ which is a contradiction to our assumption. Similarly, $\nu(E_{m+1} / E_m)=0$. To show that $\nu_m$ and $\nu_{m+1}$ are equivalent, it is suffices to show that $\nu_m \ll \nu_{m+1}$. Suppose that $\nu_{m+1}(A)=0$ for $A \in \mathcal{E}$. This implies that for each subset $B$ of $A$, we have
\begin{align*}
	\int_B f_{\Psi^{m+1}} d\nu \leq \int_A f_{\Psi^{m+1}} d\nu=\nu_{m+1}(A)=0\\
	\Rightarrow \nu\{x \in A:f_{\Psi^m}(x)\neq 0, f_{\Psi^{m+1}}(x)\neq 0\}=0.
\end{align*}
Also, $\nu\{x \in A:f_{\Psi^m}(x)\neq 0, f_{\Psi^{m+1}}(x)= 0\}=0$ as it is a subset of $E_{m+1} / E_m$. Therefore,
\begin{align*}
	\nu_m(A)=\int_C f_{\Psi^m} d\nu=\int_D f_{\Psi^m} d\nu= 0 
\end{align*} 
where $C=\{x\in A:f_{\Psi^m}=0\}$ and $D=\{x\in A:f_{\Psi^m} \neq 0\}$.
\end{proof}
\begin{theorem}
	Let $\Psi$ be a non-singular measurable transformation on the
	measure space $E$ inducing the composition operator $C_\Psi$ on the Orlicz-Lorentz space $L_{(\phi,w)}$. A necessary and sufficient condition for  $\mathcal{A}(C_\Psi)=m$ is that $m$ is the least non-negative integer such that the
	measures $\nu_m$ and $\nu_{m+1}$ are equivalent.
\end{theorem}
	
	\begin{corollary}
		Let $C_\Psi \in \mathcal{B}(L_{(\phi,w)})$. Then the
	$\mathcal{A}(C_\Psi)=0$ in each of the following situations:
		\begin{enumerate}
			\item $\Psi$ is a measure preserving.
			\item  $\Psi$ is surjective.
		\end{enumerate}
	\end{corollary}
Theorem 3.3 can be restated in the following manner:
\begin{theorem}
The	necessary and sufficient condition for a composition operator $C_\Psi$ on Orlicz-Lorentz space $L_{(\phi,w)}$ to have $\mathcal{A}(C_\Psi)=\infty$ is that the measures $\nu_m$ and $\nu_{m+1}$ can never be equivalent for any natural number $m$.
\end{theorem}

\begin{theorem}
A sufficient condition for an operator $C_\Psi \in \mathcal{B}(L_{(\phi,w)})$, induced by a measurable transformation $\Psi$, to have $\mathcal{A}(C_\Psi)< \infty$ is that $\Psi$ is a pre-positive measurable transformation.\end{theorem}

\begin{proof} Let $\Psi$ be a pre-positive measurable transformation inducing $C_\Psi$ on $L_{(\phi,w)}$. Let us suppose that $\mathcal{A}(C_\Psi) \nless \infty$, then for each natural number $m$, there exists $f_m \in \mathcal{N}(C_\Psi^m)$ such that
	 $$ \nu\left\{x \in E: f_{\psi^m} \circ \Psi^{m-1}(x)\neq 0\right\}>0.$$ 
	 Since $\Psi$ is pre-positive, so $\nu(\Psi^{-1}(E_m)>0$ where $$ \nu(\Psi^{-1}(E_m)= \nu(\left\{x \in E: (f_{\Psi^m} \circ \Psi^m) (x) \neq 0\right\}),$$ which is a contradiction to the fact that $f_{\Psi^m} \in \mathcal{N}(C_\Psi^m)$. Therefore, it follows that  $\mathcal{A}(C_\Psi)< \infty$.\end{proof}

\begin{theorem}
If the measurable transformation $\Psi$ inducing $C_\Psi \in \mathcal{B}(L_{(\phi,w)})$, is pre-positive, then 
$\mathcal{A}(C_\Psi)=0$.
\end{theorem}

\begin{theorem}
Let $C_\Psi \in \mathcal{B}(L_{(\phi,w)})$. If there exists a sequence $\{A\}_{m}$ of measurable sets such that for each $m$, $0< \nu(A_m) <\infty$, $\nu(\Psi^{-m}(A_m)) = 0$ and $\nu(\Psi^{-(m-1)} (A_m))\neq 0$, then $\mathcal{A}(C_\Psi)=\infty$.\end{theorem}
\begin{proof}Let $\{A_m\}_{m\geq 1}$ be a sequence of measurable sets satisfying $0<\nu(A_m)<\infty$, $(\Psi^{-m}(A_m)) = 0$ and $\nu(\Psi^{-(m-1)} (A_m))\neq 0$ for each $m$. For each measurable set $S$, we know that

 $$ \| \chi_S\| =\frac{1}{\phi^{-1}(\frac{1}{\nu(S)})}.$$

 Hence, for each natural number $m$, the characteristics function $ \mathcal{X}_{A_m} \in L_{(\phi,w)}$ and $\chi_{A_m} \in (\mathcal{N}(C_\Psi^m) / \mathcal{N}(C_\Psi^{m-1}))$. Therefore, $\mathcal{A}(C_\Psi)=\infty$.
\end{proof} 
\begin{theorem}
Let the measurable transformation $\Psi$ on $E$ inducing $C_\Psi$ on $L_{(\phi, w)}$ be such that the image of each measurable set is measurable. If $\mathcal{A}(C_\Psi)\nless \infty$ on $L_{(\phi, w)}$, then there exists a sequence of subsets ${A_m}$ of $E$ such that for all $m>1$,
\begin{enumerate}
	\item $0 <\nu(A_m) <\infty$,
	\item $A_m \subseteq \Psi^{m-1}(B)$ for some $B\in \mathcal{E}$,
	\item $A_m\notin \{\Psi^m(S) : S\in \mathcal{E}\ and\ \nu(S) >0\}$. 
\end{enumerate}\end{theorem}

\begin{proof}
Let $\mathcal{A}(C_\Psi)= \infty$. Then, for each positive integer $m$, we have $\mathcal{N}(C_\Psi^{m-1}) \subset \mathcal{N}(C_\Psi^m)$. This implies that there exists $f_{\Psi^m} \in \mathcal{N}(C_\Psi^m)$ such that $f_{\Psi^m} \notin \mathcal{N}(C_\Psi^{m-1})$. Take $ E_m= \left\{x \in E : f_{\Psi^m} \circ \Psi^{m-1}(x) \neq 0\right\}$. Clearly, $\nu(E_m) > 0$ and $\Psi^{m-1}(E_m)$ is measurable. Now, we show that $\nu(\Psi^{m-1}(E_m)) > 0$. Let us suppose that $\nu(\Psi^{m-1}(E_m)) = 0$. Since $\Psi$ is non-singular and $E_m \subseteq \Psi^{-(m-1)} (\Psi^{m-1} (E_m))$, we get
\[\nu(E_m)\leq \nu(\Psi^{-(m-1)}(\Psi^{m-1}(E_m)))=0,\]
which is a contradiction. Hence, $\nu(\Psi^{m-1}(E_m))>0$.

Now, since the measure $\nu$ is $\sigma$-finite, we can choose a measurable subset $A_m$ of $\Psi^{m-1}(E_m)$ such that $0 <\nu (A_m) <\infty$  where $x \in A_m$ such that $f_{\Psi^m}(x) \neq 0$. 
Next, let if possible that $A_m = \Psi^m (S)$ for some measurable set $S$ having positive measure, then
\[\nu(\{x \in S: f_{\Psi^m} \circ \Psi^m (x) = 0\}) = 0.\]
Since $f_{\Psi^m} \in \mathcal{N}(C_\Psi^m), \nu(\{x \in S: (f_{\Psi^m} \circ \Psi^m) (x)\neq 0\}) = 0$. The above two sets each having measure zero mean that $\nu(S) = 0$, this contradicts the fact that $\nu(S) > 0$. Hence the proof.\end{proof}

If we put $E = \mathbb{N}$, $\mathcal{E} = 2^{\mathbb{N}}$ and $\nu$ is the counting measure,  then the corresponding
Orlicz-Lorentz space is called as Orlicz-Lorentz sequence space and is denoted
by $l_{(\phi,\omega)}$\cite{3}. The following theorem gives the characterization of the composition
operators to have infinite ascent on the Orlicz-Lorentz sequence spaces $l_{(\phi,\omega)}$.

\begin{theorem}A necessary and sufficient condition for the composition operator $C_\Psi$ on the Orlicz-Lorentz sequence space $l_{(\phi,\omega)}$  induced by $\Psi: \mathbb{N} \to \mathbb{N}$ to have $\mathcal{A}(C_\Psi)=\infty$ is that there exists a sequence of distinct natural numbers $\langle n_m \rangle$ such that $n_m \not\in \Psi^m(\mathbb{N})$  but $n_m \in \Psi^{m-1}(\mathbb{N})$ for each $m \geq 1$.\end{theorem}

\begin{proof} The sufficient part is a direct conclusion. For the necessary part, using Theorem 3.9, we  get a sequence of non empty measurable subsets ${A_m}$ of $\mathbb{N}$ satisfying $A_m \subseteq \Psi^{m-1}(\mathbb{N})$ and $A_m \nsubseteq \Psi^m(\mathbb{N}) \subseteq \Psi^{m-1}(\mathbb{N})$. Hence we can take a sequence of distinct natural number $n_m$ satisfying $n_m \in \Psi^{m-1} (\mathbb{N})$ and $n_m \notin \Psi^m (\mathbb{N})$.\end{proof}

\begin{center}
	\section{Descent of composition operators}
\end{center}
\begin{definition} A measure space $(E, \mathcal{E},\nu)$ is said to be \textit{separable} if for every distinct pair of points $x$ and $y$ in $E$, we can find disjoint positive measurable sets $E_1$ and $E_2$ such that $x \in E_1$ and $y \in E_2$.\end{definition}


\begin{theorem}
The composition operator $C_\Psi$ on the Orlicz-Lorentz space $L_{(\phi,w)}$ has $\mathcal{D}(C_\Psi)=0$ if $C_\Psi$ is bounded away from zero on its support and $\Psi^{-1}(\mathcal{E}) = \mathcal{E}$.\end{theorem}

\begin{theorem} Suppose that $(E, \mathcal{E},\nu)$ is a separable $\sigma$-finite measure space. A necessary condition for the composition operator $C_\Psi$ on $L_{(\phi,w)}$  to have  $\mathcal{D}(C_\Psi) < \infty$ is that $\hat{\Psi_m}$ is injective for some non negative integer $m$, where $\hat{\Psi_m} = \Psi|_{\mathcal{R}(\Psi^m)}: \mathcal{R}(\Psi^m) \to \mathcal{R}(\Psi^m)$.\end{theorem}
\begin{proof}
 Suppose that  $\mathcal{D}(C_\Psi)< \infty$ on $L_{(\phi,w)}$. Let if possible for each $m$, the mapping $\hat{\Psi^m}$ is not injective,  then there exist $x_1$, $x_2$, in $E$ such that
  $$\Psi^m (x_1) \neq \Psi^m (x_2)\  \text{and} $$
  $$\Psi^{m+1}(x_1) = \Psi^{m+1}(x_2).$$
   Also, since $(E,\mathcal{E},\nu)$ is separable and $\sigma$-finite, therefore there exists two disjoint measurable sets $E_1$ and $E_2$ with non-zero finite measures containing $x(= \Psi^m(x_1))$ and $y(=\Psi^m(x_2))$, respectively.
   
    Hence $\chi_{E_1},\chi_{E_2} \in L_{(\phi,\omega)}$. Now consider the element $f_1$ of $L_{(\phi,\omega)}$ given by $$f_1 = \mathcal{X}_{X_1}-\mathcal{X}_{X_2}.$$

     Then $f_2 = C_\Psi^m f_1 \in \mathcal{R}(C_\Psi^m)$. But $f_2 \notin \mathcal{R}(C_\Psi^{m+1})$ because if  $f_2 = C_\Psi^{m+1}f_3$ for some $f_3 \in L_{(\phi,\omega)}(E)$, 
     \begin{align*}
     & 1 = f_1(x)\\
       &= f_1(\Psi^mx_1)\\
        &= C_\Psi^{m+1}f_\Psi(x_1)\\
         &= C_\Psi^{m+1}f_\Psi(x_2)\\
          &= f_2(x_2)\\
           &= C_\Psi^mf_1(x_2)\\
           &=f_1(y)\\
            &= -1.
\end{align*}
This implies that $\mathcal{R} (C_\Psi^{m+1}) \subset \mathcal{R} (C_\Psi^m)$ for each non negative integer $m$. Hence  $\mathcal{D}(C_\Psi) = \infty$.\end{proof}

\begin{theorem} Suppose that $(E,\mathcal{E},\nu)$ is a separable $\sigma$-finite measure space. Then a necessary condition for the composition operator $C_\Psi$ on $L_{(\phi,w)}$, to have the descent at most $m$ is that  $\Psi_m$ is injective.\end{theorem}

\begin{corollary}Let  $(E,\mathcal{E}, \nu)$ be the measure space such that every singleton set has positive measure and $C_\Psi \in  \mathcal{B}(L_{(\phi,\omega)})$. Then $\mathcal{D}(C_\Psi) >m $ if the mapping $\Psi|_{\mathcal{R}(\Psi^m)}: \mathcal{R}(\Psi^m) \to  \mathcal{R}(\Psi^m)$ is not injective.\end{corollary}
\begin{proof}  Let the measure space   $(E,\mathcal{E}, \nu)$ be $\sigma$-finite and separable and  $\Psi|_{\mathcal{R}(\Psi^m)}$ is not injective, then (as in Theorem 4.2), there exists measurable sets $E_1$ and $E_2$  containing $x(=\Psi^m(x_1))$ and $y(=\Psi^m(x_2))$ as the singleton sets ${x_1}$ and ${x_2}$ respectively, we obtain that $\mathcal{D}(C_\Psi) >m $.\end{proof}

\end{document}